\documentclass[10pt]{amsproc}
\usepackage{amsmath, hyperref ,cite, amssymb, color}
\newtheorem{theorem}{Theorem}[section]
\newtheorem{lemma}[theorem]{Lemma}

\theoremstyle{definition}
\newtheorem{definition}[theorem]{Definition}

\theoremstyle{remark}
\newtheorem{remark}[theorem]{Remark}

\newtheorem{property}[theorem]{Property}
\numberwithin{equation}{section}

\begin{document}
\title[Fractional Hankel wavelet transform]{Continuity of the fractional Hankel wavelet transform on the spaces of type S}

%\author{Akhilesh Prasad$^{\dagger}$}
%\address{$^{\dagger}$Department of Applied Mathematics, Indian Institute of Technology(Indian School of Mines), Dhanbad-826004, India}
%\email{apr$\_$bhu$@$yahoo.com}
\author{Kanailal Mahato}
\address{Department of  Mathematics, Institute of Science, Banaras Hindu University, Varanasi-221005, India}
\email{kanailalmahato$@$gmail.com, kanailalmahato$@$bhu.ac.in}
\date{\today}


\begin{abstract}
In this article we study the fractional Hankel transform and its inverse on certain Gel'fand-Shilov spaces of type S. The continuous fractional wavelet transform is defined involving  the  fractional Hankel transform. The continuity of fractional Hankel wavelet transform is discussed on Gel'fand-Shilov spaces of type S. This article goes further to discuss the continuity property of fractional Hankel transform and fractional Hankel wavelet transform  on the ultradifferentiable function spaces. 
\end{abstract}
\maketitle

\section{Introduction} 
In the recent years,  the continuous wavelet transform has been successfully applied in the field of signal processing, image encryption. The continuous wavelet transform of a function $f$ associated with the wavelet $\psi$ is defined by
\begin{eqnarray}
(W_{\psi}f)(b, a)=\int_{-\infty}^{\infty}f(t) \overline{\psi_{b, a}}(t)\frac{dt}{a},\nonumber
\end{eqnarray}
where $\psi_{b, a}(t)= \psi\Big(\frac{t-b}{a} \Big),$ provided the integral exists, where $a\in \mathbb{R}^+$ and $b\in \mathbb{R}$. If $f, \psi \in L^2(\mathbb{R})$, then exploiting the Parseval relation for Fourier transform, the above expression can be viewed as (see \cite{Chui, Deb}):
\begin{eqnarray}
(W_{\psi}f)(b, a)= \frac{1}{2\pi}\int_{-\infty}^{\infty} e^{ib\omega} \hat{f}(\omega) \overline{\hat{\psi}(a\omega)}d\omega,\nonumber
\end{eqnarray}
where $\hat{f}$ and $\hat{\psi}$ denotes the Fourier transform of $f$ and $\psi$ respectively. The Gel'fand-Shilov spaces were introduced  in \cite{gelfand} and studied the characterization of Fourier transform on the aforesaid spaces. Pathak \cite{pathak0}and  Holschneider \cite{hols} studied the wavelet transform involving Fourier transform, on Schwartz space $S(\mathbb{R})$. Zemanian \cite{ze}, Lee\cite{lee} and Pathak\cite{rs} described the basic properties of classical Hankel transform on the certain Gel'fand-Shilov spaces of type S. In the theory of partial differential equations, mathematical analysis the spaces of type S play an important role as an intermediate spaces between those of $C^{\infty}$ and of the analytic functions. The main purpose this article is to study the fractional Hankel transform and continuous wavelet transform associated with fractional Hankel transform on Gel'fand-Shilov spaces of type S.

The fractional Hankel transform (FrHT), which is a generalization of the usual Hankel transform and depends on a parameter $\theta$, has been the focus of many researcher as it has a wide range of applications in the field of seismology, optics, signal processing, solving problems involving cylindrical boundaries. The fractional  Hankel transform  $\mathcal{H}_{ \nu, \mu}^{\theta}$ of a function $f$ of order $\nu \geq -\frac{1}{2}$ depends on an arbitrary real parameter $\mu$ and $\theta (0< \theta < \pi)$, is defined  by (see \cite{kerr, prasad2, prasad4, torre}):
\begin{eqnarray}
(\mathcal H_{ \nu,\mu}^{\theta} f )(\omega)= \tilde{f}^{\theta}(\omega) =\displaystyle \int _{0}^\infty  K^{\theta}(t,\omega) f(t) dt, \label{eq:1.1}
\end{eqnarray}
where,
\begin{eqnarray}
&&\quad K^{\theta}(t,\omega) = 
\begin{cases}
C_{\nu, \mu, \theta}e^{\frac{i}{2}(t^2+\omega^2)\cot \theta}(t\omega  csc\theta)^{-\mu}J_{\nu}(t \omega   \csc\theta) t^{1+2\mu}, ~\theta \neq n\pi\\
(t \omega)^{-\mu}J_{\nu}(t \omega ) t^{1+2\mu},\quad \quad \quad  \quad \quad \quad \quad \quad \quad \quad \quad \quad \theta = \frac{\pi}{2}\\
\delta(t-\omega),\quad \quad \quad \quad \quad\quad \quad  \quad \quad \quad \quad \quad\quad\quad \theta= n \pi, ~n \in \mathbb{Z}
\end{cases}
 \label{eq:1.2} 
\end{eqnarray}
where  $C_{\nu, \mu, \theta}= \frac{e^{i(1+\nu)(\theta-\frac{\pi}{2})}}{(\sin \theta)^{1+\mu}}$.\\
The inverse of \eqref{eq:1.1}  given  as follows:
\begin{eqnarray}
 f(t)= ((\mathcal H_{\nu,\mu}^{-\theta} )\tilde{f}^{\theta} )(t)= \displaystyle \int _{0}^\infty  K^{-\theta}(\omega, t) \tilde{f}^{\theta}(\omega) d\omega, \label{eq:1.3}
\end{eqnarray}  
where $K^{-\theta}(\omega, t)$ is same as $\overline{{K^{\theta}}}(\omega, t) $.
\par Let the space   $L_{\nu,\mu}^{p}(I)$ contains  of all those measurable functions  $f$ on $I = (0, \infty)$ such that the integral $\displaystyle \int_{0}^{\infty} \vert f(t)\vert^{p} t^{\mu+\nu+1} dt $ exist and is finite. Also let $L^{\infty}(I)$ be the collection of almost everywhere bounded integrable functions. Hence endowed with the norm
\begin{eqnarray}
\vert \vert f \vert \vert _{L_{\nu,\mu}^{p}}=
\left\{
\begin{array}{l}
{\left(\displaystyle \int_{0}^{\infty} \vert f(t)\vert^{p} t^{\mu+\nu+1} dt\right)}^{\frac{1}{p}}, 1\leq p < \infty, \mu, \nu \in \mathbb{R}\\
  \displaystyle ess \sup_{t \in I}\vert f(t) \vert, ~~~ \quad \quad p = \infty. 
\end{array}
\right. \label{eq:1.4}
\end{eqnarray}
\textbf{Parseval's relation:} It is easy to see that for the operator $\mathcal{H}_{\nu, \mu}^{\theta} $, under certain conditions,
\begin{eqnarray*}
\int_{0}^{\infty} f(t)\overline{g(t)}t^{1+2\mu}dt = \int_{0}^{\infty}\tilde{f}^{\theta}(\omega) \overline{\tilde{g}^{\theta}(\omega)}\omega^{1+2\mu}d\omega.
\end{eqnarray*} 

To define the fractional Hankel translation\cite{prasad2, prasad4, hamio, mahato1} $\tau_{t}^{\theta} $ of a function $\psi \in L_{\nu, \mu}^1(I)$, we need to introduce $D_{\nu,\mu}^{\theta}$, which is defined by:
\begin{eqnarray}
D_{\nu,\mu}^{\theta}(t,\omega,z)&=&C_{\nu,\mu, -\theta} \displaystyle \int_{0}^{\infty} {(zs \csc \theta)}^{-\mu}J_{\nu}(zs \csc \theta)e^{-\frac{i}{2}(z^2+t^2+\omega^2) \cot \theta}\label{eq:1.5}\\
&& \times {(ts \csc \theta)}^{-\mu} J_{\nu}(ts \csc \theta){(\omega s \csc \theta)}^{-\mu}J_{\nu}(\omega s \csc \theta)\nonumber \\
&& \times s^{1+3\mu-\nu}ds,\nonumber 
\end{eqnarray} 
provided the integral exist.

%Let  $\psi \in L_{\nu, \mu}^1(I)$. As per \cite{prasad2, prasad4, hamio, mahato1}
 The fractional Hankel translation \cite{hirs} $ \tau_{t}^{\theta}$of  $\psi$ is given by
\begin{eqnarray}
(\tau_{t}^{\theta}\psi)(\omega)& =& \psi^{\theta}(t,\omega)\nonumber\\
&=& C_{\nu, \mu, \theta} \displaystyle \int_{0}^{\infty} \psi(z)D_{\nu,\mu}^{\theta}(t,\omega, z)e^{\frac{i}{2}z^2 \cot \theta} z^{\mu+\nu+1} dz. \label{eq:1.6}
\end{eqnarray}
Wavelets are considered to be the set of elements constructed from translation and  dilation of a single function $\psi \in L^2(\mathbb{R})$ \cite{pathak0, Chui, Deb}. In the similar way \cite{upadhyay, pathak} introduced the Bessel wavelet and the fractional Bessel wavelet by \cite{mahato1, prasad4, pathak, prasad1} as $\psi_{b, a, \theta}$, which is defined as below:
\begin{eqnarray}
\psi_{b, a, \theta}(t)&=&\mathcal{D}_a(\tau_{b}^{\theta}\psi)(t)=\mathcal{D}_{a}\psi^{\theta}(b, t)=a^{-2\mu-2} e^{\frac{i}{2}(\frac{1}{a^2}-1)t^2 \cot \theta} e^{\frac{i}{2}(\frac{1}{a^2}+1)b^2 \cot \theta}\nonumber\\
&& \times \psi^{\theta}(b/a, t/a), ~~b\geq 0, a>0,
\end{eqnarray}
where $\mathcal{D}_a$ represents the dilation of a function.

As per \cite{mahato1, pathak, prasad4, Chui, Deb, hols}, the fractional wavelet transform $W_{\psi}^{\theta}$ of $f\in L^2_{\nu, \mu}(I)$ associated with the wavelet $\psi \in L^2_{\nu, \mu}(I)$ defined by means of the integral transform
\begin{eqnarray}
(W_{\psi}^{\theta}f)(b, a)= \displaystyle \int_{0}^{\infty}f(t) \overline{\psi_{b,a,\theta}(t)} t^{1+2\mu}dt.\label{eq:1.8}
\end{eqnarray}
Now exploiting Parseval's relation and following \cite{mahato1, prasad4, pathak}, the above expression can be rewritten as
\begin{eqnarray}
(W_{\psi}^{\theta}f)(b, a)&=& \frac{1}{C_{\nu, \mu, -\theta}}\displaystyle \int_{0}^{\infty} K^{-\theta}(\omega, b) {(a\omega)}^{\mu-\nu} e^{\frac{i}{2}a^2\omega^2 \cot \theta} \tilde{f}^{\theta}(\omega) \nonumber \\
&& \times \overline{\mathcal{H}_{\nu, \mu}^{\theta}(z^{\nu-\mu}e^{-\frac{i}{2}z^2 \cot \theta} \psi(z))(a\omega)} d\omega\nonumber\\
&=&  \frac{1}{C_{\nu, \mu, -\theta}} \mathcal{H}_{\nu, \mu}^{-\theta} \Big[ {(a\omega)}^{\mu-\nu} e^{-\frac{i}{2}a^2\omega^2 \cot \theta}\tilde{f}^{\theta}(\omega)\nonumber\\
&& \times  \overline{\mathcal{H}_{\nu, \mu}^{\theta} (z^{\nu-\mu}e^{-\frac{i}{2}z^2 \cot \theta } \psi(z))}(a\omega)\Big](b).\label{eq:1.9}
\end{eqnarray}
According to \cite{lee, rs}, we now introduce the certain   Gel'fand-Shilov spaces of type S on which  the fractional Hankel transform \eqref{eq:1.1} and the fractional Hankel wavlet transform \eqref{eq:1.9} can be studied. Let us recall the definitions of these spaces.
\begin{definition}
%Let $f$ be an infinitely differentiable function defined on $I=(0, \infty)$.
 The space $\mathbb{H}_{1,  \alpha, A}(I)$ consists of infinitely differentiable functions $f$ on $I=(0, \infty)$ satisfying the inequality 
\begin{eqnarray}
\left \vert x^k{(x^{-1}D_x)}^q e^{\pm \frac{i}{2}x^2 \cot \theta} x^{\mu-\nu}f(x) \right \vert \leq C_{q}^{\nu, \mu} {(A+\delta)}^k k^{k\alpha}, ~~\forall k , q \in \mathbb{N}_{0},\label{eq:1.10}
\end{eqnarray}
where the constants $A$ and $ C_{q}^{\nu, \mu}$ depends on $f$ and $\alpha, \delta\geq 0$ are arbitrary constants and the  norms are given by
\begin{eqnarray}\label{eq:1.11}
\vert \vert f\vert \vert_{q}^{\nu, \mu, \theta}=  \sup_{0<x<\infty} \frac{\left \vert x^k{(x^{-1}D_x)}^q e^{\pm \frac{i}{2}x^2 \cot \theta} x^{\mu-\nu}f(x)\right \vert}{{(A+\delta)}^k k^{k\alpha}} < \infty.
\end{eqnarray}
\end{definition}

\begin{definition}\label{defi:1.2}
The function $f \in \mathbb{H}^{2,  \beta, B}(I)$ iff 
\begin{eqnarray}
\left \vert x^k{(x^{-1}D_x)}^q e^{\pm \frac{i}{2}x^2 \cot \theta} x^{\mu-\nu}f(x) \right \vert \leq C_{k}^{\nu, \mu} {(B+\sigma)}^q q^{q\beta}, ~~\forall k , q \in \mathbb{N}_{0},
\end{eqnarray}
where the constants $B, C_{k}^{\nu, \mu}$ depend on $f$ and $\sigma, \beta\geq 0$ is an arbitrary constant. In this space the norms are given by
\begin{eqnarray}
\vert \vert f\vert \vert_{k}^{\nu, \mu, \theta}=  \sup_{0<x<\infty} \frac{\left \vert x^k{(x^{-1}D_x)}^q e^{\pm \frac{i}{2}x^2 \cot \theta} x^{\mu-\nu}f(x)\right \vert}{{(B+\sigma)}^q q^{q\beta}} < \infty.
\end{eqnarray}
\end{definition}

\begin{definition}
The space $\mathbb{H}^{\beta, B}_{\alpha, A}(I)$ is defined as follows: $f \in \mathbb{H}^{\beta, B}_{\alpha, A}(I)$ if and only if  
\begin{eqnarray}\label{eq:1.14}
\left \vert x^k{(x^{-1}D_x)}^q e^{\pm \frac{i}{2}x^2 \cot \theta} x^{\mu-\nu}f(x) \right \vert \leq C^{\nu, \mu} {(A+\delta)}^k k^{k \alpha}{(B+\sigma)}^q q^{q\beta},
\end{eqnarray}
$\forall k , q \in \mathbb{N}_{0},$ where the constants $A, B, C^{\nu, \mu}$ depend on $f$ and $\alpha, \beta, \delta, \sigma\geq 0$ are arbitrary constants. We introduce the norms in the space $\mathbb{H}^{\beta, B}_{\alpha, A}(I)$ as follows:
\begin{eqnarray}
\vert \vert f\vert \vert^{\nu, \mu, \theta}=  \sup_{0<x<\infty} \frac{\left \vert x^k{(x^{-1}D_x)}^q e^{\pm \frac{i}{2}x^2 \cot \theta} x^{\mu-\nu}f(x)\right \vert}{{(A+\delta)}^k k^{k \alpha}{(B+\sigma)}^q q^{q\beta}} < \infty.
\end{eqnarray}
\end{definition}

Now we need to introduce the following types of test function spaces \cite{pathak0}
\begin{definition}
The space $\mathbb{H}_{1, \tilde{\alpha}, \tilde{A}}(I\times I)$, $\tilde{\alpha}=(\alpha_1, \alpha_2)$, $\alpha_1, \alpha_2 \geq 0$ and $\tilde{A}=(A_1, A_2)$, is defined as the collection of all smooth functions $f(b, a) \in I\times I$, such that for all $l, k, p, q\in \mathbb{N}_0,$
\begin{eqnarray}
&& \sup_{a, b} \vert a^l b^k {(a^{-1}D_a)}^p {(b^{-1}D_b)}^q e^{\pm \frac{i}{2}b^2 \cot \theta}b^{\mu-\nu}f(b, a) \vert \nonumber\\
&\leq &C_{p, q}^{\nu, \mu} {(A_1+\delta_1)}^{l} l^{l\alpha_1} {(A_2+\delta_2)}^{k} k^{k\alpha_2},
\end{eqnarray}
where the constants $A_1, A_2$ and $C_{p, q}^{\nu, \mu}$ depending on $f$ and $\delta_1, \delta_2 \geq 0$ be arbitrary constants.
\end{definition}

\begin{definition}
The space $\mathbb{H}^{2, \tilde{\beta}, \tilde{B}}(I\times I)$, $\tilde{\beta}=(\beta_1, \beta_2)$, $\beta_1, \beta_2 \geq 0$ and $\tilde{B}=(B_1, B_2)$, is defined as the space of all smooth functions $f(b, a)\in I\times I$, such that for all $l, k, p, q\in \mathbb{N}_0,$
\begin{eqnarray}
&& \sup_{a, b} \vert a^l b^k {(a^{-1}D_a)}^p {(b^{-1}D_b)}^q e^{\pm \frac{i}{2}b^2 \cot \theta}b^{\mu-\nu}f(b, a) \vert \nonumber\\
&\leq &C_{l, k}^{\nu, \mu} {(B_1+\sigma_1)}^{p} p^{p\beta_1} {(B_2+\sigma_2)}^{q} q^{q\beta_2},
\end{eqnarray}
where $\sigma_1, \sigma_2 \geq 0$ be arbitrary constants and $B_1, B_2, C_{l, k}^{\nu, \mu}$ be the constants depends on $f$.
\end{definition}

\begin{definition}
The space $\mathbb{H}_{\tilde{\alpha}, \tilde{A}}^{\tilde{\beta}, \tilde{B}}(I\times I)$, $\tilde{\alpha}=(\alpha_1, \alpha_2), \tilde{\beta}=(\beta_1, \beta_2)$, $\alpha_1, \alpha_2, \beta_1, \beta_2 \geq 0$ and $\tilde{A}=(A_1, A_2), \tilde{B}=(B_1, B_2)$, is defined as the space of all infinitely differentiable  functions $f(b, a)\in I\times I$, such that for all $l, k, p, q\in \mathbb{N}_0,$
\begin{eqnarray}
&& \sup_{a, b} \vert a^l b^k {(a^{-1}D_a)}^p {(b^{-1}D_b)}^q e^{\pm \frac{i}{2}b^2 \cot \theta}b^{\mu-\nu}f(b, a) \vert \nonumber\\
&\leq &C^{\nu, \mu} {(A_1+\delta_1)}^{l}  {(A_2+\delta_2)}^{k}  {(B_1+\sigma_1)}^{p}  {(B_2+\sigma_2)}^{q} l^{l\alpha_1}k^{k\alpha_2} p^{p\beta_1}q^{q\beta_2},
\end{eqnarray}
where the constants $A_1, A_2, B_1, B_2, C^{\nu, \mu}$ depending on $f$ and $\delta_1, \delta_2, \sigma_1, \sigma_2 \geq 0$ are arbitrary constants.
\end{definition}

From \cite{prasad2, torre} we have the differential operator $M_{\nu, \mu, \theta}$ as:\\
$M_{\nu, \mu, \theta}=-e^{-\frac{i}{2}x^2 \cot \theta} x^{\nu-\mu}D_{x}e^{\frac{i}{2}x^2 \cot \theta} x^{\mu-\nu}.$ We shall need the following Lemma in the proof of the Theorem \ref{th2.1}.
\begin{lemma}\label{lemma:1.7}
Suppose that $\nu \geq -\frac{1}{2}, \mu, \theta$ as above and $q, k \in \mathbb{N}_{0}$. For $\psi \in \mathbb{W}_{\nu, \mu}^{\theta}$, then we obtain
$$(i) ~M_{\nu+k-1, \mu, \theta}...M_{\nu, \mu, \theta}\psi(x)={(-1)}^{k} x^{\nu-\mu+k}e^{-\frac{i}{2}x^2 \cot \theta}{(x^{-1}D_x)}^{k}e^{\frac{i}{2}x^2 \cot \theta} x^{\mu-\nu}\psi(x),$$
$(ii) ~M_{\nu+q-1, \mu, \theta}...M_{\nu, \mu, \theta}( \mathcal{H}_{\nu, \mu}^{-\theta} \psi)(y)={(\csc \theta e^{i(\theta-\pi/2)})}^q (\mathcal{H}_{\nu+q, \mu}^{-\theta} x^q \psi)(y)$,\\
$(iii) ~\mathcal{H}_{\nu+q+k, \mu}^{-\theta}(x^q M_{\nu+k-1, \mu, -\theta} ...M_{\nu, \mu, -\theta}\psi)(y)={\Big(y\csc \theta e^{-i(\theta-\pi/2)}\Big)}^k (\mathcal{H}_{\nu+q, \mu}^{-\theta}x^q \psi)(y).$
\end{lemma}
\begin{proof}
Since,
\begin{eqnarray}
M_{\nu, \mu, \theta}&=&-e^{-\frac{i}{2}x^2 \cot \theta} x^{\nu-\mu}D_{x}e^{\frac{i}{2}x^2 \cot \theta} x^{\mu-\nu}\nonumber\\
M_{\nu+1, \mu, \theta}M_{\nu, \mu, \theta}\psi(x)&=& x^{\nu-\mu+2}e^{-\frac{i}{2}x^2 \cot \theta} {(x^{-1}D_x)}^2 x^{\mu-\nu}e^{\frac{i}{2}x^2 \cot \theta}\psi(x).\nonumber
\end{eqnarray}
Proceeding in this way $k^{th}$ times, we get the required result $(i)$.\\
Now to prove $(ii),$ we have
\begin{eqnarray}
&&M_{\nu+q-1, \mu, \theta}...M_{\nu, \mu, \theta}( \mathcal{H}_{\nu, \mu}^{-\theta} \psi)(y)\nonumber\\
&=& {(-1)}^q y^{\nu-\mu+q} e^{-\frac{i}{2}y^2 \cot \theta} {(y^{-1}D_y)}^q y^{\mu-\nu}e^{\frac{i}{2}y^2 \cot \theta} C_{\nu, \mu, -\theta}\int_{0}^{\infty}{(xy\csc \theta)}^{-\mu}\nonumber\\
&& \times J_{\nu}(xy \csc \theta) e^{-\frac{i}{2}(x^2+y^2)\cot \theta}x^{1+2\mu} \psi (x) dx.\nonumber
\end{eqnarray}
Now exploiting the formula ${(x^{-1}D_x)}^m[x^{-n}J_{n}(x)]={(-1)}^m x^{-n-m}J_{n+m}(x),$ where $m, n$ being positive integers, the above expression becomes
\begin{eqnarray}
&&C_{\nu, \mu, -\theta} \int_{0}^{\infty} {(xy\csc \theta)}^{-\mu}J_{\nu}(xy \csc \theta)e^{-\frac{i}{2}(x^2+y^2)\cot \theta}x^{1+2\mu} {(x\csc \theta)}^q \psi(x)dx\nonumber\\
&=& {(\csc \theta e^{i(\theta-\pi/2)})}^q (\mathcal{H}_{\nu+q, \mu}^{-\theta}x^q\psi)(y) \nonumber.
\end{eqnarray}
This completes the proof of $(ii)$.\\
Using integration by parts we get
\begin{eqnarray}
&& (\mathcal{H}_{\nu+q+1, \mu}^{-\theta}x^qM_{\nu, \mu, -\theta}\psi)(y) \nonumber\\
&=& - C_{\nu+q+1, \mu, -\theta} \int_{0}^{\infty} e^{-\frac{i}{2}y^2 \cot \theta} {(y\csc \theta)}^{-\mu} x^{\nu+q+1}J_{\nu+q+1}(xy \csc \theta) D_{x}x^{\mu-\nu}\nonumber\\
&& \times e^{-\frac{i}{2}x^2 \cot \theta} \psi(x)dx\nonumber
\end{eqnarray}
Using the formula $D_{x}(x^nJ_n(x))=x^n J_{n-1}(x)$, in the above equation, then the above expression can be expressed as
\begin{eqnarray}
&& C_{\nu+q+1, \mu, -\theta} \int_{0}^{\infty} e^{-\frac{i}{2}y^2 \cot \theta} {(y\csc \theta)}^{-\mu} D_x[x^{\nu+q+1}J_{\nu+q+1}(xy \csc \theta)]\nonumber\\
 && \times x^{\mu-\nu} e^{-\frac{i}{2}x^2 \cot \theta} \psi(x)dx\nonumber\\
 &=& y\csc \theta e^{-i(\theta-\pi/2)} (\mathcal{H}_{\nu+q, \mu}^{-\theta}x^q \psi)(y).\nonumber
\end{eqnarray}
Continuing $k^{th}$ times in the similar manner, we get the required result $(iii)$.
\end{proof}
We shall make use of the following inequality in our present study (see \cite[pp. 265]{fried}):
\begin{eqnarray}
{(m+n)}^{q(m+n)}\leq m^{m q} n^{n q}e^{m q}e^{n q}.\label{eq:1.19}
\end{eqnarray}
We shall need the following Leibnitz formula from \cite[p.134]{ze},
\begin{eqnarray}
&&{(t^{-1}D_t)}^{n}[e^{-\frac{i}{2}t^2 \cot \theta} t^{\mu-\nu}f(t)g(t)] \nonumber\\
&=&\displaystyle \sum_{r=0}^{n}\binom{n}{r}{(t^{-1}D_t)}^{r}[e^{-\frac{i}{2}t^2 \cot \theta} t^{\mu-\nu}f(t)] {(t^{-1}D_t)}^{n-r}[g(t)].\label{eq:1.20}
\end{eqnarray}

This article consists of four sections. Section 1 is introductory part, in which several properties and fundamental definitions are given. In section 2, continuous fractional Hankel  transform$(\mathcal{H}_{\nu, \mu}^{\theta})$ and its inverse $(\mathcal{H}_{\nu, \mu}^{-\theta})$ is studied on certain Gel'fand-Shilov spaces of type S. Section 3 is devoted to the study of continuous fractional Hankel wavelet transform in the space of certain Gel'fand-Shilov spaces of type S. In section 4, fractional Hankel transform and wavelet transform associated with fractional Hankel transform is investigated on ultradifferentiable function spaces.
 
\section{The fractional Hankel transform $\mathcal{H}_{\nu, \mu}^{\theta}$ on the spaces of type S}
In this section we consider the mapping properties of the fractional Hankel transform $(\mathcal{H}_{\nu, \mu}^{\theta})$ and inverse fractional Hankel transform $(\mathcal{H}_{\nu, \mu}^{-\theta})$ on the spaces $\mathbb{H}_{1, \alpha, A}(I), \mathbb{H}^{2, \beta, B}(I)$ and $\mathbb{H}_{\alpha, A}^{\beta, B}(I)$.
\begin{theorem}
The inverse fractional Hankel transform $(\mathcal{H}_{\nu, \mu}^{-\theta})$ is a continuous linear mapping from $\mathbb{H}_{1, \alpha, A}(I)$ into $\mathbb{H}^{2, 2\alpha, A^2 {(2 e)}^{2\alpha}}(I),$ for $\nu \geq -\frac{1}{2}.$ \label{th2.1}
\end{theorem}
\begin{proof}
Exploiting Lemma \ref{lemma:1.7} $(ii)$ and $(iii)$ we obtain
\begin{eqnarray}
&&M_{\nu+q-1, \mu, \theta}...M_{\nu, \mu, \theta} (\mathcal{H}_{\nu, \mu}^{-\theta}\psi)(y)\nonumber\\
&=& {(\csc \theta e^{i(\theta-\pi/2)})}^q (\mathcal{H}_{\nu+q, \mu}^{-\theta} x^q \psi)(y)\nonumber\\
&=& {(\csc \theta )}^{q-k}y^{-k}{(  e^{i(\theta-\pi/2)})}^{q+k}  \mathcal{H}_{\nu+q+k, \mu}^{-\theta}(x^q M_{\nu+k-1, \mu, -\theta}...M_{\nu, \mu, -\theta})(y)\nonumber\\
&=& {(\csc \theta )}^{q-k}y^{-k}{(  e^{i(\theta-\pi/2)})}^{q+k}  C_{\nu+q+k, \mu, -\theta} \int_{0}^{\infty} {(xy \csc \theta)}^{-\mu}J_{\nu+q+k}(xy\csc \theta)\nonumber\\
&& \times {(-1)}^k e^{-\frac{i}{2}y^2 \cot \theta} x^{1+q+\mu+\nu+k} {(x^{-1}D_x)}^k [e^{-\frac{i}{2}x^2 \cot \theta} x^{\mu-\nu} \psi(x)] dx. \nonumber
\end{eqnarray}
Thus,
\begin{eqnarray}
&& {(-1)}^q y^k {(y^{-1}D_y)}^q e^{\frac{i}{2}y^2\cot \theta} y^{\mu-\nu}(\mathcal{H}_{\nu, \mu}^{-\theta}\psi)(y)\nonumber\\
&=& {(-1)}^k{(\csc \theta)}^{\nu+2q-k-\mu} C_{\nu, \mu, -\theta} \int_{0}^{\infty} {(xy \csc \theta)}^{-\nu-q} J_{\nu+q+k}(xy\csc \theta) x^{1+k+2\nu+2q}\nonumber\\
&& \times {(x^{-1}D_x)}^{k} [e^{-\frac{i}{2}x^2 \cot \theta} x^{\mu-\nu}\psi(x)] dx. \label{eq:2.1}
\end{eqnarray}
Now, we choose $m$ be any natural number in such a way that $m\geq 1+2\nu$; upon taking $n=m+2q+k$ and use the fact that $\vert x^{-\nu-q}J_{\nu+q+k}(x) \vert\leq C$. Then writing the integral on the right hand side of \eqref{eq:2.1} as  a sum of two integrals from 0 to 1 and 1 to $\infty$ and using  \eqref{eq:1.10}, \eqref{eq:1.19},  we have
\begin{eqnarray}
&&\vert y^k {(y^{-1}D_y)}^q [e^{\frac{i}{2}y^2\cot \theta} y^{\mu-\nu}(\mathcal{H}_{\nu, \mu}^{-\theta}\psi)(y)]\vert \nonumber\\
&\leq & C\Big[ \sup \vert y^{2q+k}{(y^{-1}D_y)}^k [e^{-\frac{i}{2}y^2\cot \theta} y^{\mu-\nu} \psi(y) ]\vert \nonumber\\
&& + \sup \vert y^{n+2}{(y^{-1}D_y)}^k [e^{-\frac{i}{2}y^2\cot \theta} y^{\mu-\nu} \psi(y)] \vert \Big]\nonumber\\
&\leq & C\Big[ C_{k}^{'} {(A+\delta)}^{2q+k} {(2q+k)}^{\alpha(2q+k)} + C_{k}^{''} {(A+\delta)}^{m+2q+k+2} \nonumber\\
&& \times {(m+2q+k+2)}^{\alpha(m+2q+k+2)} \Big]\nonumber\\
&\leq & C_1\Big[ 1+ {(A+\delta)}^{m+2q+k+2}{(m+2q+k+2)}^{\alpha(m+2q+k+2)} \Big]\nonumber\\
&\leq & C_{1}^{'} {(A^2+\delta)}^q {(m+k+2)}^{\alpha(m+k+2)}e^{\alpha(m+k+2)} {(2q)}^{\alpha 2q} e^{\alpha 2q}\nonumber\\
&\leq & C_{2}^{'} {(A^2{(2e)}^{2\alpha}+\delta^{'})}^{q} q^{\alpha 2q}.\nonumber
\end{eqnarray}
This completes the proof.
\end{proof}
\begin{remark}\label{remark2.2}
Let $\nu \geq-1/2,$ then the fractional Hankel transform $(\mathcal{H}_{\nu, \mu}^{\theta})$ is a continuous linear mapping from $\mathbb{H}_{1, \alpha, A}(I)$ into $\mathbb{H}^{2, 2\alpha, A^2 {(2 e)}^{2\alpha}}(I)$.
\end{remark}

\begin{definition}
Let $\hat{\mathbb{H}}^{2, \beta, B}(I)$ be the space of all functions $f\in \mathbb{H}^{2, \beta, B}(I)$ satisfying the condition
\begin{eqnarray}
\sup_{0\leq r\leq q} C_{k+r}^{\nu, \mu}= C_{k}^{* \nu, \mu},
\end{eqnarray}
where $C_{k}^{* \nu, \mu}$ are constants restraining the $f's$ in $\mathbb{H}^{2, \beta, B}(I)$.
\end{definition}
\begin{theorem}\label{th2.4}
The inverse fractional Hankel transform ($\mathcal{H}_{\nu, \mu}^{-\theta}$) defined by \eqref{eq:1.3} is a continuous linear mapping from $\hat{\mathbb{H}}^{2, \beta, B}(I)$ into $\mathbb{H}_{1, \beta, B}(I)$, for $\nu\geq -1/2$.
\end{theorem}
\begin{proof}
Following as in the proof of the theorem \ref{th2.1} and using \eqref{eq:1.19} and Definition \ref{defi:1.2}, we have
\begin{eqnarray}
&& \vert  y^k {(y^{-1}D_y)}^q[ e^{\frac{i}{2}y^2\cot \theta} y^{\mu-\nu}(\mathcal{H}_{\nu, \mu}^{-\theta}\psi)(y)]\vert \nonumber\\
&\leq & D \Big[ \int_{0}^{1} x^{1+k+2\nu+2q} \vert {(xy \csc \theta)}^{-\nu-q} J_{\nu+q+k}(xy\csc \theta) \vert \nonumber\\
&& \times \vert {(x^{-1}D_x)}^{k} [e^{-\frac{i}{2}x^2 \cot \theta} x^{\mu-\nu}\psi(x)] \vert dx\nonumber\\
&& + \int_{1}^{\infty} x^{1+k+2\nu+2q+2} \vert {(xy \csc \theta)}^{-\nu-q} J_{\nu+q+k}(xy\csc \theta) \vert \nonumber\\
&& \times \vert {(x^{-1}D_x)}^{k}[ e^{-\frac{i}{2}x^2 \cot \theta} x^{\mu-\nu}\psi(x)] x^{-2} \vert dx\Big]\nonumber\\
&\leq & D[C_{1+k+2\nu+2q}^{\nu, \mu}+C_{1+k+2\nu+2q+2}^{\nu, \mu}]{(B+\sigma)}^{k}k^{k\beta}\nonumber\\
&\leq & C_{q}^{*~\nu, \mu}{(B+\sigma)}^{k}k^{k\beta}. \nonumber
\end{eqnarray}
Which completes the proof.
\end{proof}

\begin{remark}\label{remark2.4}
The  fractional Hankel transform ($\mathcal{H}_{\nu, \mu}^{\theta}$)  is a continuous linear mapping from $\hat{\mathbb{H}}^{2, \beta, B}(I)$ into $\mathbb{H}_{1, \beta, B}(I)$, for $\nu\geq -1/2$.
\end{remark}

\begin{theorem}
For $\nu \geq -1/2,$ the inverse fractional Hankel transform $(\mathcal{H}_{\nu, \mu}^{-\theta})$ is a continuous linear mapping from $\mathcal{H}_{\alpha, A}^{\beta, B}(I)$ into $\mathcal{H}_{\alpha+\beta, ABe^{\alpha}}^{2\alpha, A^2 {(2e)}^{2\alpha}}(I)$.
\end{theorem}
\begin{proof}
In this case we obtain from \eqref{eq:2.1} and \eqref{eq:1.14},
\begin{eqnarray}
&&\vert y^k {(y^{-1}D_y)}^q [e^{\frac{i}{2}y^2\cot \theta} y^{\mu-\nu}(\mathcal{H}_{\nu, \mu}^{-\theta}\psi)(y)]\vert \nonumber\\
&\leq & C\Big[ \sup \vert y^{2q+k}{(y^{-1}D_y)}^k [e^{-\frac{i}{2}y^2\cot \theta} y^{\mu-\nu} \psi(y) ]\vert \nonumber\\
&& + \sup \vert y^{n+2}{(y^{-1}D_y)}^k [e^{-\frac{i}{2}y^2\cot \theta} y^{\mu-\nu} \psi(y)] \vert \Big]\nonumber\\
&\leq & C\Big[ C_{1}^{\nu, \mu} {(A+\delta)}^{2q+k} {(2q+k)}^{\alpha(2q+k)} {(B+\sigma)}^{k} k^{k\beta} + C_{2}^{\nu, \mu} {(A+\delta)}^{m+2q+k+2} \nonumber\\
&& \times {(m+2q+k+2)}^{\alpha(m+2q+k+2)} {(B+\sigma)}^{k} k^{k\beta}\Big]\nonumber\\
&\leq & C {(A+\delta)}^{m+2q+k+2} {(m+2q+k+2)}^{\alpha(m+2q+k+2)} {(B+\sigma)}^{k} k^{k\beta}.\nonumber
\end{eqnarray}
Now using \eqref{eq:1.19} in the above equation, then the above estimate can be rewritten as
\begin{eqnarray}
&&\vert y^k {(y^{-1}D_y)}^q [e^{\frac{i}{2}y^2\cot \theta} y^{\mu-\nu}(\mathcal{H}_{\nu, \mu}^{-\theta}\psi)(y)]\vert \nonumber\\
&\leq & C {(B+\sigma)}^{k} {(A+\delta)}^k k^{k\beta}  {(A+\delta)}^{2q+m+2} {(2 q)}^{2\alpha q} {(k+m+2)}^{\alpha (k+m+2)} \nonumber\\
&& \times e^{2\alpha q} e^{\alpha(k+m+2)}\nonumber\\
&\leq & C^{'} {(AB+\delta_1)}^{k} k^{k(\alpha+\beta)} {(A^2+\delta_2)}^{q} 2^{2\alpha q} q^{2\alpha q} e^{2\alpha q} e^{\alpha k}\nonumber\\
&\leq & C^{''} {(ABe^{\alpha}+\delta_1^{'})}^{k} k^{k(\alpha+\beta)} {\left(A^2 {(2e)}^{2\alpha}+\delta_2^{''}\right)}^{q} q^{2\alpha  q}.\nonumber
\end{eqnarray}
Hence the theorem proved.
\end{proof}
\begin{remark}\label{remark2.6}
For $\nu \geq -1/2,$ the fractional Hankel transform $(\mathcal{H}_{\nu, \mu}^{\theta})$ is a continuous linear mapping from $\mathbb{H}_{\alpha, A}^{\beta, B}(I)$ into $\mathbb{H}_{\alpha+\beta, ABe^{\alpha}}^{2\alpha, A^2 {(2e)}^{2\alpha}}(I)$.
\end{remark}

\section{The fractional wavelet  transform on the spaces of type S}
In this section we study the wavelet transform on the spaces of type $S$. In order to discuss the continuity of fractional wavelet transform $W_{\psi}^{\theta}$ on the aforesaid function spaces, we need to introduce the following function space.

\begin{definition}\label{def:3.1}
The space $\mathbb{W}_{\nu, \mu, \theta}(I),$ consists of all those wavelets $\psi$, $\forall n \in \mathbb{N}_0$ and $\rho \in \mathbb{R}$ which satisfy 
\begin{eqnarray}
{(t^{-1}D_t)}^{n}[t^{\mu-\nu}e^{\frac{i}{2}t^2 \cot \theta}\overline{\mathcal{H}_{\nu, \mu}^{\theta}(z^{\nu-\mu}e^{-\frac{i}{2}z^2 \cot \theta}\psi)(t)}] < D^{n, \rho}{(1+t)}^{\rho-n},
\end{eqnarray}
where $D^{n, \rho}$ is constant.
\end{definition}

\begin{theorem}
Suppose  $\psi$ be the wavelet taken from $\mathbb{W}_{\nu, \mu, \theta}(I)$. The continuous fractional wavelet transform $W_{\psi}^{\theta}$ is a continuous linear mapping from $\mathbb{H}_{1, \alpha, A}(I)$ into $\mathbb{H}_{1, \tilde{\alpha}, \tilde{A}}(I\times I)$, where $\tilde{\alpha}=(0, 2\alpha)$ and $\tilde{A}=\left(a, A^2{(2e)}^{2\alpha}+a^2\right).$
\end{theorem}
\begin{proof}
From the definition of $W_{\psi}^{\theta}$ from \eqref{eq:1.9} and using \eqref{eq:2.1}, we obtain
\begin{eqnarray}
&&\vert b^k {(b^{-1}D_b)}^q e^{\frac{i}{2}b^2 \cot \theta} b^{\mu-\nu} (W_{\psi}^{\theta}f)(b, a)\vert \nonumber\\
&= & \Big\vert {(\csc \theta)}^{2q+\nu-\mu-k} \int_{0}^{\infty} \omega^{1+2\nu+2q+k} {(b\omega \csc \theta)}^{-\nu-q}J_{\nu+q+k}(b\omega \csc \theta)\nonumber\\
&& \times {(\omega^{-1}D_{\omega})}^k\Big[e^{-\frac{i}{2}\omega^2 \cot \theta} {(a\omega)}^{\mu-\nu}e^{\frac{i}{2}{a^2 \omega^2} \cot \theta} \overline{\mathcal{H}_{\nu, \mu}^{\theta}(z^{\nu-\mu}e^{-\frac{i}{2}z^2 \cot \theta}\psi)(a\omega)}\nonumber\\
&& \times \omega^{\mu-\nu}\tilde{f}^{\theta}(\omega)\Big]d\omega\Big\vert.\nonumber
\end{eqnarray}
Using the fact that $\vert x^{-\nu-q}J_{\nu+q+k}(x)\vert\leq C$ and in viewing \eqref{eq:1.20}, the above relation becomes
\begin{eqnarray}
&&\vert b^k {(b^{-1}D_b)}^q e^{\frac{i}{2}b^2 \cot \theta} b^{\mu-\nu} (W_{\psi}^{\theta}f)(b, a)\vert \nonumber\\
&\leq & C \vert{(\csc \theta)}^{2q+\nu-\mu-k}\vert  \int_{0}^{\infty} \omega^{1+2\nu+2q+k} \sum_{r=0}^{k}\binom{k}{r}{(\omega^{-1}D_{\omega})}^r\Big[{(a\omega)}^{\mu-\nu}e^{\frac{i}{2}{a^2 \omega^2} \cot \theta}\nonumber\\
&& \times \overline{\mathcal{H}_{\nu, \mu}^{\theta}(z^{\nu-\mu}e^{-\frac{i}{2}z^2 \cot \theta}\psi)(a\omega)}~\Big] {(\omega^{-1}D_{\omega})}^{k-r}\Big[ e^{-\frac{i}{2}\omega^2 \cot \theta}\omega^{\mu-\nu}\tilde{f}^{\theta}(\omega)\Big] d\omega.
\end{eqnarray}
Therefore,
\begin{eqnarray}
&&\vert b^k {(a^{-1}D_a)}^p{(b^{-1}D_b)}^q e^{\frac{i}{2}b^2 \cot \theta} b^{\mu-\nu} (W_{\psi}^{\theta}f)(b, a)\vert \nonumber\\
&\leq & C\sum_{r=0}^{k}\binom{k}{r} \int_{0}^{\infty} \omega^{1+2\nu+2q+k}{(a^{-1}D_a)}^p{(\omega^{-1}D_{\omega})}^r\Big[{(a\omega)}^{\mu-\nu}e^{\frac{i}{2}{a^2 \omega^2} \cot \theta}\nonumber\\
&& \times \overline{\mathcal{H}_{\nu, \mu}^{\theta}(z^{\nu-\mu}e^{-\frac{i}{2}z^2 \cot \theta}\psi)(a\omega)}~\Big] {(\omega^{-1}D_{\omega})}^{k-r}\Big[ e^{-\frac{i}{2}\omega^2 \cot \theta}\omega^{\mu-\nu}\tilde{f}^{\theta}(\omega)\Big] d\omega. \label{eq:3.3}
\end{eqnarray}
Exploiting the definition \ref{def:3.1} for $t=a\omega$ we obtain
\begin{eqnarray}
 && \left\vert {(a^{-1}D_{a})}^{p}{(\omega^{-1}D_{\omega})}^{r} \Big[ {(a\omega)}^{\mu-\nu}e^{\frac{i}{2}a^2\omega^2 \cot \theta}  \overline{\mathcal{H}_{\nu, \mu}^{\theta}(z^{\nu-\mu} e^{-\frac{i}{2}z^2 \cot \theta}\psi)}(a\omega)\Big]\right \vert \nonumber\\
&= & \left\vert a^{2r} \omega^{2p}{(t^{-1}D_{t})}^{p+r} \Big[ {t}^{\mu-\nu}e^{\frac{i}{2}t^2 \cot \theta}  \overline{\mathcal{H}_{\nu, \mu}^{\theta}(z^{\nu-\mu} e^{-\frac{i}{2}z^2 \cot \theta}\psi_1)}(t)\Big]\right \vert \nonumber\\
&\leq & \left\vert a^{2r} \omega^{2p} D^{p+r, \rho_1}{(1+t)}^{\rho_1-p-r}\right \vert \nonumber\\
&\leq & \left\vert a^{2r} \omega^{2p} D^{p+r, \rho_1} {(1+a\omega)}^{\rho_1-p-r} \right \vert \nonumber\\
& \leq & \vert a^{2r} \omega^{2p} D^{p+r, \rho_1} {(1+a)}^{\rho_1-p-r}{(1+\omega)}^{\rho_1-p-r}\vert. \label{eq:3.4}
\end{eqnarray}
Using \eqref{eq:3.4} in \eqref{eq:3.3} and assuming $\nu_1$ be any positive integer such that $\nu_1\geq 1+2\nu$, we have
\begin{eqnarray}
&&\vert a^l b^k {(a^{-1}D_a)}^p{(b^{-1}D_b)}^q e^{\frac{i}{2}b^2 \cot \theta} b^{\mu-\nu} (W_{\psi}^{\theta}f)(b, a)\vert \nonumber\\
&\leq & C\sum_{r=0}^{k}\binom{k}{r} a^{2r} \int_{0}^{\infty} \omega^{\nu_1+2q+2p+k} {(1+a)}^{\rho_1-r-p}{(1+\omega)}^{\rho_1-r-p+s} \nonumber\\
&& \times  {(\omega^{-1}D_{\omega})}^{k-r}\Big[ e^{-\frac{i}{2}\omega^2 \cot \theta}\omega^{\mu-\nu}\tilde{f}^{\theta}(\omega)\Big] \frac{1}{{(1+\omega)}^s}d\omega \nonumber\\
&\leq & C\sum_{r=0}^{k} \sum_{n=0}^{\rho_1-r-p+s}\binom{k}{r}\binom{\rho_1-r-p+s}{n} a^{2r+l} {(1+a)}^{\rho_1-r-p}\nonumber\\
&& \times \sup \left \vert \omega^{\nu_1+2q+2p+k+n}{(\omega^{-1}D_{\omega})}^{k-r}\Big[ e^{-\frac{i}{2}\omega^2 \cot \theta}\omega^{\mu-\nu}\tilde{f}^{\theta}(\omega)\Big] \right\vert \int_{0}^{\infty}\frac{1}{{(1+\omega)}^s}d\omega. \label{eq:3.5}
\end{eqnarray}
Exploiting the remark \ref{remark2.2} and \eqref{eq:1.11}, the right hand-side of the above estimate becomes

\begin{eqnarray}
&\leq & C\sum_{r=0}^{k} \sum_{n=0}^{\rho_1-r-p+s}\binom{k}{r}\binom{\rho_1-r-p+s}{n} a^{2r+l} {(1+a)}^{\rho_1-r-p} {\Big(A^2{(2e)}^{2\alpha}+\delta_2\Big)}^ {k-r} \nonumber\\
  && \times {(k-r)} ^{2\alpha(k-r)} \max_{0\leq r\leq k} \vert \vert \tilde{f}^{\theta}\vert \vert^{\nu, \mu, \theta}_{k-r}\nonumber\\
  &\leq & C \sum_{n=0}^{\rho_1-p+s}\binom{\rho_1-p+s}{n}a^l {(1+a)}^{\rho_1-p}\sum_{r=0}^{k}\binom{k}{r} {(a^2)^r}{\Big(A^2{(2e)}^{2\alpha}+\delta_2\Big)}^{k-r}k^{k 2\alpha}\nonumber\\
  && \times \max_{0\leq r\leq k} \vert \vert \tilde{f}^{\theta}\vert \vert^{\nu, \mu, \theta}_{k-r}\nonumber\\
 &\leq & C^{*} {(1+a)}^{\rho_1-p}a^l {\Big(A^2{(2e)}^{2\alpha}+a^2+\delta_2^{'}\Big)}^k k^{k 2\alpha} \max_{0\leq r\leq k} \vert \vert \tilde{f}^{\theta}\vert \vert^{\nu, \mu, \theta}_{k-r}\nonumber.
\end{eqnarray}
This completes the proof.
\end{proof}

\begin{theorem}
Let  $\psi \in \mathbb{W}_{\nu, \mu, \theta}(I)$. The continuous fractional wavelet transform $W_{\psi}^{\theta}$ is a continuous linear mapping from $\hat{\mathbb{H}}^{2, \beta, B}(I)$ into $\hat{\mathbb{H}}^{2, \tilde{\beta}, \tilde{B}}(I\times I)$, where $\tilde{\beta}=(2\beta, 2\beta)$ and $\tilde{B}=\left(\frac{B^2}{a}{(2e)}^{2\beta}, B^2{e}^{2\beta}\right).$
\end{theorem}
\begin{proof}
From the estimate \eqref{eq:3.5} and using \eqref{eq:1.19}, we obtain
\begin{eqnarray}
&&\vert a^l b^k {(a^{-1}D_a)}^p{(b^{-1}D_b)}^q e^{\frac{i}{2}b^2 \cot \theta} b^{\mu-\nu} (W_{\psi}^{\theta}f)(b, a)\vert \nonumber\\
&\leq & C\sum_{r=0}^{k} \sum_{n=0}^{\rho_1-r-p+s}\binom{k}{r}\binom{\rho_1-r-p+s}{n} a^{2r+l} {(1+a)}^{\rho_1-r-p}\nonumber\\
&& \times \sup \left \vert \omega^{\nu_1+2q+2p+k+n}{(\omega^{-1}D_{\omega})}^{k-r}\Big[ e^{-\frac{i}{2}\omega^2 \cot \theta}\omega^{\mu-\nu}\tilde{f}^{\theta}(\omega)\Big] \right\vert \int_{0}^{\infty}\frac{1}{{(1+\omega)}^s}d\omega\nonumber\\
&\leq & C\sum_{r=0}^{k} \sum_{n=0}^{\rho_1-p+s}\binom{k}{r}\binom{\rho_1-p+s}{n} a^{2r+l} {(1+a)}^{\rho_1-p} {(B+\sigma)}^{\nu_1+2q+2p+k+n}\nonumber\\
&&\times {(\nu_1+2q+2p+k+n)}^{\beta(\nu_1+2q+2p+k+n)} \max \vert \vert \tilde{f}^{\theta}\vert \vert _{(\nu_1+2q+2p+k+n)}^{\nu, \mu, \theta}\nonumber\\
&\leq & C^{'} \sum_{r=0}^{k} \sum_{n=0}^{\rho_1-p+s}\binom{k}{r}\binom{\rho_1-p+s}{n} a^{2r+l} {(1+a)}^{\rho_1-p}{(B+\sigma)}^{2p} {(B+\sigma)}^{2q} {(2p)}^{\beta 2p} \nonumber\\
&& \times {(\nu_1+2q+k+n)}^{\beta(\nu_1+2q+k+n)} e^{\beta 2p} e^{\beta(\nu_1+2q+k+n)}\max \vert \vert \tilde{f}^{\theta}\vert \vert _{(\nu_1+2q+2p+k+n)}^{\nu, \mu, \theta}\nonumber\\
&\leq & C^{'} \sum_{r=0}^{k} \sum_{n=0}^{\rho_1-p+s}\binom{k}{r}\binom{\rho_1-p+s}{n} a^{2r+l} {(B^2/a+\sigma_1)}^p {(B+\sigma_2)}^{2q} p^{p 2\beta} q^{q2\beta}2^{\beta 2p} \nonumber\\
&& \times e^{2p\beta} e^{2q\beta}\max \vert \vert \tilde{f}^{\theta}\vert \vert _{(\nu_1+2q+2p+k+n)}^{\nu, \mu, \theta}\nonumber\\
&\leq & C^{'} \sum_{r=0}^{k} \sum_{n=0}^{\rho_1-p+s}\binom{k}{r}\binom{\rho_1-p+s}{n} a^{2r+l} {\left(\frac{B^2}{a}{(2e)}^{2\beta} +\sigma_1^{'}\right) }^p {(B^2e^{2\beta}+\sigma_2^{'})}^q\nonumber\\
&& \times p^{p 2\beta} q^{q2\beta}\max \vert \vert \tilde{f}^{\theta}\vert \vert _{(\nu_1+2q+2p+k+n)}^{\nu, \mu, \theta}.\nonumber
\end{eqnarray}
Hence the theorem proved.
\end{proof}

\begin{theorem}
Let  $\psi \in \mathbb{W}_{\nu, \mu, \theta}(I)$. The continuous fractional wavelet transform $W_{\psi}^{\theta}$ is a continuous linear mapping from $\mathbb{H}^{ \beta, B}_{\alpha, A}(I)$ into $\mathbb{H}^{ \tilde{\beta}, \tilde{B}}_{\tilde{\alpha}, \tilde{A}}(I\times I)$, where $\tilde{\alpha}=(0, 3\alpha+\beta), \tilde{\beta}=\Big(2(\alpha+\beta), 2(\alpha+\beta)\Big)$ and $\tilde{A}=\left(a, (A^2 {(2e)}^{2\alpha}+a^2)ABe^{3\alpha+2\beta}\right)$ and $\tilde{B}=\Big(\frac{1}{a}A^2B^22^{2(\alpha+\beta)}e^{4\alpha+2\beta}, A^2B^2e^{6\alpha+4\beta}2^{2(\alpha+\beta)}\Big).$
\end{theorem}
\begin{proof}
Proceeding as in the proof of above theorem and in viewing the remark \ref{remark2.6}, we have
\begin{eqnarray}
&&\vert a^l b^k {(a^{-1}D_a)}^p{(b^{-1}D_b)}^q e^{\frac{i}{2}b^2 \cot \theta} b^{\mu-\nu} (W_{\psi}^{\theta}f)(b, a)\vert \nonumber\\
&\leq & C\sum_{r=0}^{k} \sum_{n=0}^{\rho_1-r-p+s}\binom{k}{r}\binom{\rho_1-r-p+s}{n} a^{2r+l} {(1+a)}^{\rho_1-r-p}\nonumber\\
&& \times \sup \left \vert \omega^{\nu_1+2q+2p+k+n}{(\omega^{-1}D_{\omega})}^{k-r}\Big[ e^{-\frac{i}{2}\omega^2 \cot \theta}\omega^{\mu-\nu}\tilde{f}^{\theta}(\omega)\Big] \right\vert \int_{0}^{\infty}\frac{1}{{(1+\omega)}^s}d\omega\nonumber\\
&\leq & C\sum_{r=0}^{k} \sum_{n=0}^{\rho_1-p+s}\binom{k}{r}\binom{\rho_1-p+s}{n} a^{2r+l} {(1+a)}^{\rho_1-p} {(ABe^{\alpha}+\delta)}^{(\nu_1+2q+2p+k+n)} \nonumber\\
&&\times {(\nu_1+2q+2p+k+n)}^{(\alpha+\beta)(\nu_1+2q+2p+k+n)} {\Big(A^2{(2e)}^{2\alpha}+\delta_2\Big)}^{k-r} {(k-r)}^{2\alpha(k-r)}\nonumber\\
&& \times \max \vert \vert f\vert \vert^{\nu, \mu, \theta}.
\end{eqnarray}
Exploiting the relation \eqref{eq:1.19}, the above estimate can be rewritten as
\begin{eqnarray}
&\leq & C^{*}  \sum_{n=0}^{\rho_1-p+s} \binom{\rho_1-p+s}{n} a^l {(1+a)}^{\rho_1-p} \sum_{r=0}^{k}\binom{k}{r} {(a^2)}^r {\Big(A^2{(2e)}^{2\alpha}+\delta_2\Big)}^{k-r}\nonumber\\
&& \times {(ABe^{\alpha}+\delta)}^{(\nu_1+2q+2p+k+n)}{(2p)}^{2p(\alpha+\beta)} {(\nu_1+2q+k+n)}^{(\alpha+\beta){(\nu_1+2q+k+n})}\nonumber
\end{eqnarray}
\begin{eqnarray}
&&\times e^{(\alpha+\beta)2p} e^{(\alpha+\beta)(\nu_1+2q+k+n)}k^{2k\alpha}\max \vert \vert f\vert \vert^{\nu, \mu, \theta}\nonumber\\
&\leq &C^{*}  \sum_{n=0}^{\rho_1-p+s} \binom{\rho_1-p+s}{n} a^l {(1+a)}^{\rho_1-p}  {\Big(A^2{(2e)}^{2\alpha}+a^2+\delta_2\Big)}^k {(2p)}^{2p(\alpha+\beta)} {e}^{2p(\alpha+\beta)}\nonumber\\
&& \times {(ABe^{\alpha}+\delta)}^{(\nu_1+2q+2p+k+n)} e^{(\alpha+\beta)(\nu_1+2q+k+n)} {e}^{2q(\alpha+\beta)} {e}^{(\alpha+\beta)(\nu_1+k+n)}\nonumber\\
&&\times {(2q)}^{2q(\alpha+\beta)} {(\nu_1+k+n)}^{(\alpha+\beta)(\nu_1+k+n)} k^{2k\alpha}\max \vert \vert f\vert \vert^{\nu, \mu, \theta}\nonumber\\
&\leq & C_1\sum_{n=0}^{\rho_1-p+s} \binom{\rho_1-p+s}{n} a^l {(1+a)}^{\rho_1-p} {\Big(A^2{(2e)}^{2\alpha}+a^2+\delta_2\Big)}^k {(2p)}^{2p(\alpha+\beta)}\nonumber\\
&& \times {(ABe^{\alpha}+\delta)}^{(\nu_1+2q+2p+k+n)}  e^{2(\alpha+\beta)k} k^{k(3\alpha+\beta)} e^{2p(\alpha+\beta)} e^{4q(\alpha+\beta)} {(2q)}^{2q(\alpha+\beta)}\nonumber\\
&\leq & C_2 \sum_{n=0}^{\rho_1-p+s}\binom{\rho_1-p+s}{n} a^l {\left[ \Big(A^2{(2e)}^{2\alpha}+a^2\Big)ABe^{3\alpha+2\beta} +\delta_3\right]}^k k^{k(3\alpha+\beta)}\nonumber\\
&& \times {\left(\frac{1}{a}A^2B^22^{2(\alpha+\beta)}e^{4\alpha+2\beta} +\delta_4\right)}^p p^{p2(\alpha+\beta)} {\left( A^2B^2e^{6\alpha+4\beta} 2^{2(\alpha+\beta)}+\delta_5 \right)}^q q^{q2(\alpha+\beta)}.\nonumber
\end{eqnarray}
This completes the proof of the theorem.
\end{proof}

\section{Fractional Hankel transform on ultradifferentiable function spaces}
In this section we discuss the fractional Hankel transform on spaces more general in previous sections \cite{rs, duran, pandey, marrero}. Assume that $\{ \xi_k\}_{k=0}^{\infty}$ and $\{ \eta_q\}_{q=0} ^ {\infty}$ be two arbitrary sequences of positive numbers possesses the following properties:
\begin{property}\label{property4.1}
\hspace*{.1cm}
\begin{itemize}
\item[[1]] $\xi_k^2 \leq \xi_{k-1}\xi_{k+1}, \forall k \in \mathbb{N}_0,$
\item[[2]] $\xi_k \xi_l \leq \xi_0 \xi_{k+l}, \forall k, l \in \mathbb{N}_0,$
\item[[3]] $\xi_k \leq R H^k \displaystyle \min_{0\leq l\leq k} \xi_l \xi_{k-l},  \forall k, l \in \mathbb{N}_0, R>0, H>0,$
\item[[4]]  $\xi_{k+1}\leq RH^k \xi_k, \forall k \in \mathbb{N}_0, R>0, H>0,$
\item[[5]] $\displaystyle \sum_{j=0}^{\infty}\frac{\xi_j}{\xi_{j+1}}<\infty.$
\end{itemize} 
\end{property}
From the above property [1], we have
\begin{eqnarray}
\frac{\xi_{k}}{\xi_{k+1}}\leq \frac{\xi_{k-1}}{\xi_{k}}\leq \frac{\xi_{k-2}}{\xi_{k-1}}...\leq \frac{\xi_{0}}{\xi_{1}},\nonumber
\end{eqnarray}
and
\begin{eqnarray}
\xi_{k-r}&=&\frac{\xi_{k-r}}{\xi_{k-r+1}} \frac{\xi_{k-r+1}}{\xi_{k-r+2}}...\frac{\xi_{k-1}}{\xi_{k}}\xi_k \nonumber\\
&\leq & {\left(\frac{\xi_{0}}{\xi_{1}}\right)}^r \xi_k. \label{eq:4.1}
\end{eqnarray}
In the very similar way we obtain
\begin{eqnarray}
\eta_{q-r}&\leq & {\left(\frac{\eta_{0}}{\eta_{1}}\right)}^r \eta_q.
\end{eqnarray}
We now introduce the following types of function spaces\cite{rs}.
\begin{definition}
Let $\{ \xi_k\}_{k=0}^{\infty}$ and $\{ \eta_q\}_{q=0} ^ {\infty}$ be two any sequences of positive numbers. An infinitely differentiable complex valued function $f\in \mathbb{H}_{1, \xi_k, A}(I) $ if and only if
\begin{eqnarray}
\left \vert x^k{(x^{-1}D_x)}^q e^{\pm \frac{i}{2}x^2 \cot \theta} x^{\mu-\nu}f(x) \right \vert \leq C_{q}^{\nu, \mu} {(A+\delta)}^k \xi_k, ~~\forall k , q \in \mathbb{N}_{0},
\end{eqnarray}
for some positive constants $A, C_{q}^{\nu, \mu}$ depending on $f$; and $f$ belongs to the space $\mathbb{H}^{2, \eta_q, B}(I)$ if and only if
\begin{eqnarray}
\left \vert x^k{(x^{-1}D_x)}^q e^{\pm \frac{i}{2}x^2 \cot \theta} x^{\mu-\nu}f(x) \right \vert \leq C_{k}^{\nu, \mu} {(B+\sigma)}^q \eta_q, ~~\forall k , q \in \mathbb{N}_{0},
\end{eqnarray}
for some positive constants $B$ and $C_{k}^{\nu, \mu}$ depending on $f$; and the function $f$ is said to be in the space $\mathbb{H}_{\xi_k, A}^{\eta_q, B}(I)$ if and only if 
\begin{eqnarray}\label{eq:4.5}
\left \vert x^k{(x^{-1}D_x)}^q e^{\pm \frac{i}{2}x^2 \cot \theta} x^{\mu-\nu}f(x) \right \vert \leq C^{\nu, \mu} {(A+\delta)}^k \xi_k {(B+\sigma)}^q \eta_q,
\end{eqnarray}
$\forall k , q \in \mathbb{N}_{0},$ where $ A, B, C^{\nu, \mu} $ are certain positive constants dependent on $f$.
\end{definition}
The elements of the spaces $\mathbb{H}_{1, \xi_k, A}(I), \mathbb{H}^{2, \eta_q, B}(I)$ and $\mathbb{H}_{\xi_k, A}^{ \eta_q, B}(I)$ are known as ultradifferentiable functions \cite{koma, rs, matsu, rodino}. \\
We shall need similar types of function spaces of two variables.

\begin{definition}
The space $\mathbb{H}_{1, \xi_{ml+nk}, \tilde{A}}(I\times I), \tilde{A}=(A_1, A_2)$ is defined to the collection of all infinitely differentiable functions $f(b, a)$ satisfying, for all $l, k, m, n, p, q \in \mathbb{N}_0$,
\begin{eqnarray}
&& \sup_{a, b} \vert a^l b^k {(a^{-1}D_a)}^p {(b^{-1}D_b)}^q e^{\pm \frac{i}{2}b^2 \cot \theta}b^{\mu-\nu}f(b, a) \vert \nonumber\\
&\leq &C_{p, q}^{\nu, \mu} {(A_1+\delta_1)}^{l}  {(A_2+\delta_2)}^{k} \xi_{ml+nk},
\end{eqnarray}
where the arbitrary constants $A_1, A_2, C_{p, q}^{\nu, \mu}$ depends on $f$.
\end{definition}

\begin{definition}
The space $\mathbb{H}^{2, \eta_{sp+tq}, \tilde{B}}(I\times I), \tilde{B}=(B_1, B_2)$  is defined to the collection of all functions $f(b, a)\in C^{\infty}(I\times I)$ satisfying, for all $l, k, s, t, p, q \in \mathbb{N}_0$,
\begin{eqnarray}
&& \sup_{a, b} \vert a^l b^k {(a^{-1}D_a)}^p {(b^{-1}D_b)}^q e^{\pm \frac{i}{2}b^2 \cot \theta}b^{\mu-\nu}f(b, a) \vert \nonumber\\
&\leq & C_{l, k}^{\nu, \mu} {(B_1+\sigma_1)}^{p}  {(B_2+\sigma_2)}^{q} \eta_{sp+tq},
\end{eqnarray}
where the arbitrary constants $B_1, B_2, C_{l, k}^{\nu, \mu}$ depends on $f$.
\end{definition}

\begin{definition}
The space $\mathbb{H}_{\xi_{ml+nk}, \eta_{gl+hk}, \tilde{A} }^{\xi_{cp+ dq}, \eta_{sp+tq}, \tilde{B}}(I\times I), \tilde{A}=(A_1, A_2), \tilde{B}=(B_1, B_2)$  is defined to the collection of all infinitely differentiable functions $f(b, a)$ satisfying, for all $l, k, g, h, s, t, p, q, c, d \in \mathbb{N}_0$,
\begin{eqnarray}
&& \sup_{a, b} \vert a^l b^k {(a^{-1}D_a)}^p {(b^{-1}D_b)}^q e^{\pm \frac{i}{2}b^2 \cot \theta}b^{\mu-\nu}f(b, a) \vert \nonumber\\
&\leq & C^{\nu, \mu} {(A_1+\delta_1)}^{l}{(A_2+\delta_2)}^{k} {(B_1+\sigma_1)}^{p}  {(B_2+\sigma_2)}^{q} \xi_{ml+nk} \eta_{gl+hk} \eta_{sp+tq} \xi_{cp+ dq},\nonumber
\end{eqnarray}
where the arbitrary constants $A_1, A_2, B_1, B_2$ and $C^{\nu, \mu}$ depends on $f$.
\end{definition}

\begin{theorem}\label{th4.6}
If $\{\xi_k\}$ and $\{\eta_q\}$ be the sequences satisfies the property \ref{property4.1}, the inverse fractional Hankel transform $(\mathcal{H}_{\nu, \mu}^{-\theta})$ is a continuous linear mapping from the space $\mathbb{H}_{ \xi_k, A}^{\eta_q, B}(I)$ into $\mathbb{H}^{\xi^2_{q}, B_1}_{\xi_k \eta_k, A_1},$ where $A_1=ABH^2, B_1= A^2H^6.$
\end{theorem}
\begin{proof}
Following the procedure of the proof of the Theorem \ref{th2.1} and using property \ref{property4.1} [3] and in viewing \eqref{eq:4.5}, we have
\begin{eqnarray}
&&\vert y^k {(y^{-1}D_y)}^q [e^{\frac{i}{2}y^2\cot \theta} y^{\mu-\nu}(\mathcal{H}_{\nu, \mu}^{-\theta}\psi)(y)]\vert \nonumber\\
&\leq & C\Big[ \sup \vert y^{2q+k}{(y^{-1}D_y)}^k [e^{-\frac{i}{2}y^2\cot \theta} y^{\mu-\nu} \psi(y) ]\vert \nonumber\\
&& + \sup \vert y^{m+2q+k+2}{(y^{-1}D_y)}^k [e^{-\frac{i}{2}y^2\cot \theta} y^{\mu-\nu} \psi(y)] \vert \Big]\nonumber\\
&\leq & C\Big[ C_1^{\nu, \mu}{(A+\delta)}^{2q+k}\xi_{2q+k} +C_2^{\nu, \mu} {(A+\delta)}^{m+2q+k+2}\xi_{m+2q+k+2}\Big]{(B+\sigma)}^{k}\eta_k\nonumber\\
&\leq & C^{'} {(B+\sigma)}^{k}\eta_k {(H(A+\delta))}^{2q+k}\xi_{2q+k} [1+{(A+\delta)}^{m+2}RH^{m+2}\xi_{m+2}]\nonumber\\
&\leq & C^{''} {(ABH^2+\delta_2)}^k \xi_k \eta_k {(A^2H^6+\delta_3)}^q \eta^2_{q}.\nonumber
\end{eqnarray}
This completes the proof.
\end{proof}

\begin{remark}
Let $\{\xi_k\}$  and $\{\eta_q\}$ be the sequences satisfies the property \ref{property4.1}, the  fractional Hankel transform $(\mathcal{H}_{\nu, \mu}^{\theta})$ is a continuous linear mapping from the space $\mathbb{H}_{ \xi_k, A}^{\eta_q, B}(I)$ into $\mathbb{H}^{\xi^2_{q}, B_1}_{\xi_k \eta_k, A_1},$ where $A_1=ABH^2, B_1= A^2H^6.$
\end{remark}

\begin{theorem}\label{th4.8}
If $\{\xi_k\}$ be the sequence satisfies the property \ref{property4.1} then for $\nu\geq -1/2$, $\mathcal{H}_{\nu, \mu}^{-\theta}$ is a continuous linear mapping from $\mathbb{H}_{1, \xi_k, A}(I)$ into $\mathbb{H}^{2, \xi_q^2, A_1}(I)$, where $A_1=A^2H^6$.
\end{theorem}
\begin{proof}
From the above theorem, we have
\begin{eqnarray}
&&\vert y^k {(y^{-1}D_y)}^q [e^{\frac{i}{2}y^2\cot \theta} y^{\mu-\nu}(\mathcal{H}_{\nu, \mu}^{-\theta}\psi)(y)]\vert \nonumber\\
&\leq & C\Big[ \sup \vert y^{2q+k}{(y^{-1}D_y)}^k [e^{-\frac{i}{2}y^2\cot \theta} y^{\mu-\nu} \psi(y) ]\vert \nonumber\\
&& + \sup \vert y^{m+2q+k+2}{(y^{-1}D_y)}^k [e^{-\frac{i}{2}y^2\cot \theta} y^{\mu-\nu} \psi(y)] \vert \Big]\nonumber\\
&\leq & C\Big[ C_{k}^{\nu, \mu} {(A+\delta)}^{2q+k}\xi_{2q+k} + D_k^{\nu, \mu} {(A+\delta)}^{m+2q+k+2}\xi_{m+2q+k+2}\Big]\nonumber\\
&\leq & C^{'} {(A^2H^6+\delta_2)}^q \xi^2_q.\nonumber
\end{eqnarray}
Hence the theorem proved.
\end{proof}

\begin{remark}
Let $\{\xi_k\}$ be the sequence satisfies the property \ref{property4.1} then for $\nu\geq -1/2$, $\mathcal{H}_{\nu, \mu}^{\theta}$ is a continuous linear mapping from $\mathbb{H}_{1, \xi_k, A}(I)$ into $\mathbb{H}^{2, \xi_q^2, A_1}(I)$, where $A_1=A^2H^6$.
\end{remark}

\begin{definition}
The space $\hat{\mathbb{H}}^{2, \eta_{q}, B}(I)$ be the collection of all functions  $f \in \mathbb{H}^{2, \eta_{q}, B}(I)$ satisfying the condition
\begin{equation}
\sup_{0\leq r \leq k}C_{k+r}^{\nu, \mu} =C_{k}^{'\nu, \mu},
\end{equation}
where $C_{k}^{'\nu, \mu}$ are constants restraining the $f's$ in $\mathbb{H}^{2, \eta_{q}, B}(I)$.
\end{definition}
\begin{theorem} \label{th4.11}
For $\nu\geq -1/2$ and suppose $\{\eta_q\}$ be the sequence satisfies the property \ref{property4.1} the inverse fractional Hankel $\mathcal{H}_{\nu, \mu}^{-\theta}$ is a continuous linear mapping from $\hat{\mathbb{H}}^{2, \eta_{q}, B}(I)$ into $\mathbb{H}_{1, \eta_k, B}(I)$. 
\end{theorem}
\begin{proof}
Exploiting \eqref{eq:2.1} and Theorem \ref{th2.4}, we have
\begin{eqnarray}
&&\vert y^k {(y^{-1}D_y)}^q [e^{\frac{i}{2}y^2\cot \theta} y^{\mu-\nu}(\mathcal{H}_{\nu, \mu}^{-\theta}\psi)(y)]\vert \nonumber\\
&\leq & C\Big[ C_{1+2\nu+2q+k}^{\nu, \mu}+ C_{1+2\nu+2q+k+2}\Big]{(B+\sigma)}^k \eta_k\nonumber\\
&\leq & C_q^{*} {(B+\sigma)}^k \eta_k\nonumber.
\end{eqnarray}
This completes the proof.
\end{proof}

\begin{remark}
If $\{\eta_q\}$ be the sequence satisfies the property \ref{property4.1} and  $\nu\geq -1/2$ the  fractional Hankel transform $\mathcal{H}_{\nu, \mu}^{\theta}$ is a continuous linear mapping from $\hat{\mathbb{H}}^{2, \eta_{q}, B}(I)$ into $\mathbb{H}_{1, \eta_k, B}(I)$. 
\end{remark}

\begin{theorem}
Let $\psi$ be the wavelet belongs to the space $\mathbb{W}_{\nu, \mu, \theta}(I)$. If  $\{\xi_k\}$ be the sequence satisfies the property \ref{property4.1} then fractional Hankel wavelet transform is a continuous linear mapping from $\mathbb{H}_{1, \xi_k, A}(I)$ into $\mathbb{H}_{1, \xi_k^2, \tilde{A}}(I\times I),$ where $\tilde{A}= \left(a, A^2H^6+ \frac{a^2\xi_0^2}{\xi_1^2} \right),$ for $\nu \geq -1/2.$ 
\end{theorem}
\begin{proof}
From \eqref{eq:3.5} and in viewing Theorem \ref{th4.8}, we have
\begin{eqnarray}
&&\vert a^l b^k {(a^{-1}D_a)}^p{(b^{-1}D_b)}^q e^{\frac{i}{2}b^2 \cot \theta} b^{\mu-\nu} (W_{\psi}^{\theta}f)(b, a)\vert \nonumber\\
&\leq & C\sum_{r=0}^{k} \sum_{n=0}^{\rho_1-r-p+s}\binom{k}{r}\binom{\rho_1-r-p+s}{n} a^{2r+l} {(1+a)}^{\rho_1-r-p}\nonumber\\
&& \times \sup \left \vert \omega^{\nu_1+2q+2p+k+n}{(\omega^{-1}D_{\omega})}^{k-r}\Big[ e^{-\frac{i}{2}\omega^2 \cot \theta}\omega^{\mu-\nu}\tilde{f}^{\theta}(\omega)\Big] \right\vert \int_{0}^{\infty}\frac{1}{{(1+\omega)}^s}d\omega\nonumber\\
&\leq & C\sum_{r=0}^{k} \sum_{n=0}^{\rho_1-p+s}\binom{k}{r}\binom{\rho_1-p+s}{n} a^{2r+l} {(1+a)}^{\rho_1-p} {(A^2H^6+\delta_1)}^{k-r} \nonumber\\
&& \times \xi^2_{k-r} \max_{0\leq r\leq k} \vert \vert \tilde{f}^{\theta}\vert \vert_{k-r}^{\nu, \mu}\nonumber\\
&\leq & C\sum_{r=0}^{k} \sum_{n=0}^{\rho_1-p+s}\binom{k}{r}\binom{\rho_1-p+s}{n} a^{2r+l} {(1+a)}^{\rho_1-p} {(A^2H^6+\delta_1)}^{k-r}\nonumber\\
&& \times \xi^2_{k} {(\frac{\xi_0}{\xi_1})}^{2r} \max_{0\leq r\leq k} \vert \vert \tilde{f}^{\theta}\vert \vert_{k-r}^{\nu, \mu}\nonumber\\
&\leq & C \sum_{n=0}^{\rho_1-p+s} \binom{\rho_1-p+s}{n}a^l {(1+a)}^{\rho_1-p} \sum_{r=0}^{k} \binom{k}{r} {(A^2H^6+\delta_1)}^{k-r} {\left(a^2\frac{\xi_0^2}{\xi_1^2}\right)}^{r} \nonumber\\
&&\times \xi^2_{k}\max_{0\leq r\leq k} \vert \vert \tilde{f}^{\theta}\vert \vert_{k-r}^{\nu, \mu}\nonumber\\
&\leq & C^{*} a^l {\left(A^2H^6+ \frac{a^2 \xi^2_0}{\xi^2_1}+\delta_2\right)}^{k} \xi^2_{k}\max_{0\leq r\leq k} \vert \vert \tilde{f}^{\theta}\vert \vert_{k-r}^{\nu, \mu}\nonumber.
\end{eqnarray}
This completes the proof.
\end{proof}

\begin{theorem}
Suppose $\psi$ be the wavelet taken from $\mathbb{W}_{\nu, \mu, \theta}(I)$. If  $\{\eta_q\}$ be the sequence satisfies the property \ref{property4.1} then fractional Hankel wavelet transform is a continuous linear mapping from $\hat{\mathbb{H}}^{2, \eta_q, B}(I)$ into $\hat{\mathbb{H}}^{2, \eta_{2p+2q}, \tilde{B}}(I\times I),$ where $\tilde{B}= \left(B^2/a, B^2\right),$ for $\nu \geq -1/2.$ 
\end{theorem}
\begin{proof}
Proceeding as in the proof of the earlier theorem and exploiting Theorem \ref{th4.11}, we obtain
\begin{eqnarray}
&&\vert a^l b^k {(a^{-1}D_a)}^p{(b^{-1}D_b)}^q e^{\frac{i}{2}b^2 \cot \theta} b^{\mu-\nu} (W_{\psi}^{\theta}f)(b, a)\vert \nonumber\\
&\leq & C\sum_{r=0}^{k} \sum_{n=0}^{\rho_1-r-p+s}\binom{k}{r}\binom{\rho_1-r-p+s}{n} a^{2r+l} {(1+a)}^{\rho_1-r-p}\nonumber\\
&& \times \sup \left \vert \omega^{\nu_1+2q+2p+k+n}{(\omega^{-1}D_{\omega})}^{k-r}\Big[ e^{-\frac{i}{2}\omega^2 \cot \theta}\omega^{\mu-\nu}\tilde{f}^{\theta}(\omega)\Big] \right\vert \int_{0}^{\infty}\frac{1}{{(1+\omega)}^s}d\omega\nonumber\\
& \leq & C_1 \sum_{r=0}^{k} \sum_{n=0}^{\rho_1-p+s} \binom{k}{r} \binom{\rho_1-p+s}{n} a^{2r+l} {(1+a)}^{\rho_1-p} \max \vert \vert \tilde{f}^{\theta} \vert \vert C_{k-r}^{\nu, \mu}  \nonumber\\
&&\times {(B+\sigma)}^{\nu_1+2p+2q+k+n} \eta_{\nu_1+2p+2q+k+n}\nonumber\\
&\leq & C_2 \max \vert \vert \tilde{f}^{\theta} \vert \vert {\left(\frac{B^2}{a}+\sigma_1\right)}^{p} {(B^2+\sigma_2)}^{q} \eta_{\nu_1+2p+2q+k+n} \nonumber.
\end{eqnarray}
Hence the theorem proved.
\end{proof}

\begin{theorem}
Let $\psi \in \mathbb{W}_{\nu, \mu, \theta}(I)$. If  $\{\xi_k\}$ and $\{\eta_q\}$ be the sequences satisfies the property \ref{property4.1} then fractional Hankel wavelet transform is a continuous linear mapping from $\mathbb{H}_{\xi_k, A}^{ \eta_q, B}(I)$ into $\mathbb{H}_{{\xi^2_{2k}}, \eta_k, \tilde{A}}^{\xi_{2p+2q}, \eta_{2p+2q}, \tilde{B}}(I\times I),$ where $\tilde{A}=\Big(a, ABH^2(A^2H^6+a^2\xi_0^2/\xi_1^2)\Big)$ and $\tilde{B}= ( \frac{1}{a}A^2B^2H^6, A^2B^2H^4),$ for $\nu \geq -1/2.$ 
\end{theorem}
\begin{proof}
Using Theorem \ref{th4.6} and in viewing \eqref{eq:3.5}, we see that
\begin{eqnarray}
&&\vert a^l b^k {(a^{-1}D_a)}^p{(b^{-1}D_b)}^q e^{\frac{i}{2}b^2 \cot \theta} b^{\mu-\nu} (W_{\psi}^{\theta}f)(b, a)\vert \nonumber\\
&\leq & C\sum_{r=0}^{k} \sum_{n=0}^{\rho_1-p+s}\binom{k}{r}\binom{\rho_1-p+s}{n} a^{2r+l} {(1+a)}^{\rho_1-p}\vert \vert \tilde{f}^{\theta}\vert \vert^{\nu, \mu} {(A^2H^6+\delta_2)}^{k-r} \nonumber\\
&& \times  \xi_{k-r}^2 {(ABH^2 +\delta_1)}^ {\nu_1+2p+2q+k+n} \xi_{\nu_1+2p+2q+k+n}\eta_{\nu_1+2p+2q+k+n}\nonumber\\
&\leq &C^{'} \sum_{n=0}^{\rho_1-p+s} \binom{\rho_1-p+s}{n} a^l {(1+a)}^{\rho_1-p} \sum_{r=0}^{k} \binom{k}{r} {\Big(\frac{a^2\xi_0^2}{\xi_1^2}\Big)}^r  {(A^2H^6+\delta_2)}^{k-r}\nonumber\\
&& \times \vert \vert \tilde{f}^{\theta}\vert \vert^{\nu, \mu} \xi_{k}^2 {(ABH^2 +\delta_1)}^ {\nu_1+2p+2q+k+n} \xi_{\nu_1+2p+2q+k+n}\eta_{\nu_1+2p+2q+k+n}\nonumber\\
&\leq & C_1 a^l {\Big(ABH^2(A^2H^6+a^2\xi_0^2/\xi_1^2)+\delta_3\Big)}^k {(\frac{1}{a}A^2B^2H^6+\delta_4)}^p {(A^2B^2H^4+\delta_5)}^q\nonumber\\
&&\times\vert \vert \tilde{f}^{\theta}\vert \vert^{\nu, \mu}  \xi_{k}^2 \xi_{\nu_1+2p+2q+k+n}\eta_{\nu_1+2p+2q+k+n}\nonumber.
\end{eqnarray}
Now using the inequalities [2] and [3] from the property \ref{property4.1}, the last expression can be rewritten as
\begin{eqnarray}
&\leq&C_1 a^l {\Big(ABH^2(A^2H^6+a^2\xi_0^2/\xi_1^2)+\delta_3\Big)}^k {\left(\frac{1}{a}A^2B^2H^6+\delta_4\right)}^p {(A^2B^2H^4+\delta_5)}^q\nonumber\\
&&\times \vert \vert \tilde{f}^{\theta}\vert \vert^{\nu, \mu}\xi_{2k}^2 \eta_k \xi_{2p+2q}\eta_{2p+2q}\nonumber.
\end{eqnarray}
This completes the proof.
\end{proof}

%\textbf{Conclusion:} This article is concerned with  the study of  fractional Hankel transform and inverse  fractional Hankel transform on certain Gel'fand-Shilov spaces of type S. We introduced the certain Gel'fand-Shilov  spaces of type S and discussed the continuity of fractional Hankel transform and its inverse on the aforesaid spaces. Furthermore, it is discussed the continuity properties of the fractional Hankel wavelet  transform on the spaces of type S.

\bibliographystyle{amsplain}

\end{document}